\documentclass[12 pt, reqno]{amsart}

\usepackage{amsthm,amssymb,amstext,amscd,amsfonts,amsbsy,amsxtra,latexsym,cite}
\usepackage{mathrsfs}
\usepackage{verbatim}
\usepackage[latin1]{inputenc}
\usepackage{graphicx, graphics, wrapfig}

\newcommand{\field}[1]{\mathbb{#1}}

\newcommand{\Z}{\field{Z}}
\newcommand{\R}{\field{R}}

\def\({\left(}
\def\){\right)}

\DeclareMathOperator{\Li}{Li}

\newcommand{\calC}{\mathcal{C}}
\newcommand{\calD}{\mathcal{D}}

\theoremstyle{plain}
\newtheorem{theorem}{Theorem}
\newtheorem*{theorem*}{Theorem}

\newtheorem{proposition}[theorem]{Proposition}

\newtheorem*{conjecture*}{Conjecture}

\theoremstyle{definition}

\theoremstyle{remark}

\newtheorem*{remark}{Remark}

\numberwithin{theorem}{section} \numberwithin{equation}{section}

\begin{document}

\title{Distinct parts partitions without sequences}
\author{Kathrin Bringmann}
\address{Mathematical Institute\\University of
Cologne\\ Weyertal 86-90 \\ 50931 Cologne \\Germany}
\email{kbringma@math.uni-koeln.de}
\author{Karl Mahlburg}
\address{Department of Mathematics \\
Louisiana State University \\
Baton Rouge, LA 70802\\ U.S.A.}
\email{mahlburg@math.lsu.edu}
\author{Karthik Nataraj}


\date{\today}


\thanks{The research of the first author was supported by the Alfried Krupp Prize for Young University Teachers of the Krupp Foundation and the research leading to these results has received funding from the European Research Council under the European Union's Seventh Framework Programme (FP/2007-2013) / ERC Grant agreement n. 335220 - AQSER.  The second author was supported by NSF Grant DMS-1201435.}

\begin{abstract}

Partitions without sequences of consecutive integers as parts have been studied recently by many authors, including Andrews, Holroyd, Liggett, and Romik, among others. Their results include a description of combinatorial properties, hypergeometric representations for the generating functions, and asymptotic formulas for the enumeration functions. We complete a similar investigation of partitions into distinct parts without sequences, which are of particular interest due to their relationship with the Rogers-Ramanujan identities. Our main results include a double series representation for the generating function, an asymptotic formula for the enumeration function, and several combinatorial inequalities.
\end{abstract}

\maketitle

\section{Introduction and statement of results}
\label{S:Intro}

For $k \geq 2$, a {\it $k$-sequence} in an integer partition is any $k$ consecutive integers that all occur as parts (a standard general reference for integer partitions is \cite{And98}). Note that the case $k=1$ is excluded because any part in a nonempty partition trivially forms a ``$1$-sequence''. The study of partitions {\it without} sequences was introduced by MacMahon in Chapter IV of \cite{Mac16}. Let $p_k(n)$ be the number of partitions of $n$ with no $k$-sequences, and let $p_k(m,n)$ be the number of such partitions with $m$ parts. Since the presence of a $k$-sequence in a partition also implies the presence of a $(k-1)$-sequence, MacMahon's results on page 53 of \cite{Mac16} can be stated as the following generating function for $p_2(m,n)$,
\begin{equation}
\label{E:MacMahonG2q}
G_2(z;q) := \sum_{n,m \geq 0} p_2(m,n)z^m q^n =
1 + \sum_{n \geq 1} \frac{z^n q^n \left(q^6; q^6\right)_{n-1}}{\left(1-q^n\right)\left(q^2; q^2\right)_{n-1} \left(q^3; q^3\right)_{n-1}},
\end{equation}
where the $q$-Pochhammer symbol is defined by $(a)_n=(a;q)_n := \prod_{j = 0}^{n-1} (1- aq^j).$

Partitions without $k$-sequences for arbitrary $k \geq 2$ arose more recently in the work of Holroyd, Liggett, and Romik on probabilistic bootstrap percolation models \cite{HLR04}.
These partitions were also studied by Andrews \cite{And05}, who found a
(double) $q$-hypergeometric series expansion for the generating function,
\begin{align}
\label{E:Gkdouble}
G_k(z;q) & := \sum_{n,m \geq 0} p_k(m,n) z^m q^n  \\
& = \frac{1}{(zq;q)_\infty} \sum_{r, s \geq 0} \frac{(-1)^r z^{kr + (k+1)s}\,
q^{\textstyle \frac{(k+1)k(r+s)^2}{2} + \frac{(k+1)(s+1)s}{2}}}
{\left(q^k; q^k\right)_r \left(q^{k+1}; q^{k+1}\right)_s}. \notag
\end{align}
Andrews' proof of this expression followed from the theory of $q$-difference equations. The first two authors and Lovejoy provided an alternative bijective proof \cite{BLM13}, as well as some additional combinatorial insight into Andrews' $q$-difference equations. It should also be noted that Andrews gave another separate treatment of the case $k=2$ in \cite[Theorem 4]{And05}, where he transformed MacMahon's expression \eqref{E:MacMahonG2q} in order to write $G_2(1;q)$ in terms of one of Ramanujan's famous mock theta functions \cite{Wat36} (see \cite{ARZ13, Fol14,LO13} for a sampling of other recent results on the role of mock modular forms in hypergeometric $q$-series).

In addition to the combinatorial results described above, it is also of great interest to determine the asymptotic behavior of partitions; such study dates back to Hardy and Ramanujan's famous formula ((1.41) in \cite{HR18}), which states that as $n \to \infty$,
\begin{equation}
\label{E:pnAsymp}
p(n) \sim \frac{1}{4\sqrt{3} n} e^{\pi\sqrt{\frac{2n}{3}}}.
\end{equation}
In fact, such formulas for partitions without $k$-sequences were particularly important in \cite{HLR04}, as the {\it metastability threshold} of the $k$-cross bootstrap percolation model is intimately related to asymptotic estimates of $\log(p_k(n))$.  These approximations were subsequently refined in \cite{And05}, \cite{BM10}, and \cite{GHM12}, with the most recent progress due to Kane and Rhoades \cite[Theorem 1.8]{KR14}, who proved the asymptotic formula
\begin{equation}
\label{E:pknAsymp}
p_k(n) \sim \frac{1}{2k} \left(\frac{1}{6}\left(1 - \frac{2}{k(k+1)}\right)\right)^{\frac{1}{4}} \frac{1}{n^{\frac{3}{4}}}
\exp\left(\pi\sqrt{\frac{2}{3}\left(1 - \frac{2}{k(k+1)}\right)n}\right).
\end{equation}

\begin{remark}
The exponent in this formula was first determined by Holroyd, Liggett, and Romik \cite{HLR04}, who showed that
\begin{equation}
\label{E:logpk}
\log\left(p_k(n)\right) \sim 2 \sqrt{\left(\lambda_1 - \lambda_k\right)n},
\end{equation}
where $\lambda_k := \pi^2 / (3k(k+1)).$  Note that as $k$ becomes large this expression approaches $2 \sqrt{\lambda_1} = \pi \sqrt{2n}/3$, which is the same exponent for $\log(p(n))$ seen in \eqref{E:pnAsymp}. However, the convergence regime is more intricate for the full enumeration functions, as it is not true that \eqref{E:pknAsymp} approaches \eqref{E:pnAsymp} as $k \to \infty$, even though $p_k(n) = p(n)$ for sufficiently large $k$.

We note further that the value of $\lambda_k$ was derived in \cite{HLR04} by way of the very interesting auxiliary function $f_k: [0,1] \to [0,1]$, which is defined as the unique decreasing, positive solution to the functional equation $f^k - f^{k+1} = x^k - x^{k+1}.$ Theorem 1 of \cite{HLR04} gives the evaluation
\begin{equation*}
\label{E:lambdak=int}
\lambda_k = -\int_{0}^1 \log\left(f_k(x)\right) \frac{dx}{x}.
\end{equation*}
An alternative proof of the above evaluation is given in \cite{AEPR07}, which proceeds by rewriting $\lambda_k$ as a double integral and then making a change of variables that essentially gives the integral representation of the dilogarithm function \cite{Zag07}.
\end{remark}

In this paper we consider a natural variant of MacMahon's partitions by restricting to those partitions with no $k$-sequences that only have distinct parts. Following the spirit of the results mentioned above, we provide expressions for generating functions, describe their combinatorial properties, and determine asymptotic formulas.
Let $Q_k(n)$   be the number of partitions of $n$ with no $k$-sequences or repeated parts, and define the refined enumeration function $Q_k(m,n)$ to be the number of such partitions with $m$ parts.
Furthermore, denote the generating function by
\begin{equation*}
\calC_k \left( z; q\right) := \sum_{m,n \geq 0} Q_k (m,n) z^m q^n.
\end{equation*}
If the parts are not counted, i.e. $z=1$, then we also write $\calC_k (q) := \calC_k(1;q).$

We begin by considering the combinatorics of the case $k=2$, which corresponds to those partitions into distinct parts with no sequences.
This case is analogous to MacMahon's original study of partitions with no sequences, and a similar argument (using partition conjugation) leads to a generating function much like \eqref{E:MacMahonG2q}. In fact, the series for distinct parts partitions is even simpler, as $Q_2(n)$
counts those partitions of $n$ in which each part differs by at least $2$ -- we immediately see that these are the same partitions famously studied by Rogers and Ramanujan \cite{RR19}! We therefore have
\begin{equation*}
\calC_2(z;q) = \sum_{n \geq 0} \frac{z^n q^{n^2}}{(q;q)_n},
\end{equation*}
which specializes to the corresponding products (equations (10) and (11) of \cite{RR19})
\begin{align*}
\label{E:C2RR}
\calC_2(1;q) &= \frac{1}{\left(q, q^4; q^5\right)_\infty}, \\
\calC_2(q;q) &= \frac{1}{\left(q^2, q^3; q^5\right)_\infty}. \notag
\end{align*}
\begin{remark}
The Rogers-Ramanujan identities have inspired an incredible amount of work across divergent areas of mathematics ever since their introduction more than a century ago. For a small (and by no means exhaustive!) collection of recent work, refer to \cite{BY03,Fol14,GOW14,KLRS14}.
\end{remark}

Our first result gives double hypergeometric $q$-series expressions for our new partition functions that are analogous to \eqref{E:Gkdouble}.
\begin{theorem}
\label{T:Ckdouble}
For $k \geq 2$, we have
\begin{equation*}
\calC_k \left( z; q\right) = \sum_{j,r \geq 0} \frac{(-1)^j z^{kj + r} q^{\frac{(r+kj)(r+kj+1)}2 +k \frac{j(j-1)}2}}{\left( q^k; q^k\right)_j (q;q)_{r}}.
\end{equation*}
\end{theorem}
\noindent We give two proofs of this theorem; the first uses $q$-difference equations as in \cite{And05} and \cite{BLM13}, while the second follows the bijective arguments of \cite{BLM13}.
\begin{remark}
In fact, the statement of Theorem \ref{T:Ckdouble} and equation \eqref{E:Gkdouble} also hold for the trivial case $k=1$; here the $q$-series identities are true with $G_1(q) = \calC_1(q) = 1.$
\end{remark}
We next turn to the asymptotic study of partitions without $k$-sequences or repeated parts. As in \cite{BHMV13}, we use the Constant Term Method and a Saddle Point analysis in order to determine the asymptotic behavior of $\calC_k(q)$ near $q = 1$, and then apply Ingham's Tauberian Theorem to obtain an asymptotic formula for the coefficients. Before stating our results, we introduce two auxiliary functions (see Section \ref{S:notation} for the definition of the dilogarithm), namely
\begin{align*}
g_k(u) & := -2 \pi^2 u^2 + \Li_2\left(e^{2 \pi i u}\right)
- \frac{1}{k} \Li_2\left(e^{2 \pi i k u}\right), \\
h_k(x) & := x^{k+1} - 2x + 1.
\end{align*}
We show in Proposition \ref{P:crit} that $h_k$ has a unique root $w_k \in (0,1)$, and we let $v_k$ be the point on the positive imaginary axis such that $w_k = e^{2 \pi i v_k}.$ In other words, $v_k := i \log (w_k^{-1})/(2\pi).$

\begin{theorem}
\label{T:cknAsymp}
Using the notation above, as $n \to \infty$, we have
\begin{equation*}
Q_k(n) \sim \frac{\sqrt{\pi} g_k(v_k)^{\frac{1}{4}}}{ \sqrt{-g_k''(v_k)} n^{\frac{3}{4}}} e^{2 \sqrt{g_k(v_k) n}}.
\end{equation*}
\end{theorem}

\begin{remark}

The exponent for this result can be written in a form similar to \eqref{E:logpk}. In particular,
\begin{equation*}
\log\left( Q_k(n)\right) \sim 2 \sqrt{\left(\gamma_1 - \gamma_k\right)n},
\end{equation*}
where $\displaystyle \gamma_k := -\int_0^{\frac{1}{2}} \log\left(f_k(x)\right) dx/(x(1-x)).$

We do not present the proof of this alternative expression for the exponent, as it follows directly from the arguments in Section 3 of \cite{HLR04} (with probability $(1 + q^j)^{-1}$ for the analogous event $C_j$). Furthermore, the values of $\gamma_k$ do not simplify as cleanly as the $\lambda_k$, as the integral does not reduce to a dilogarithm evaluation. However, it is true that $\gamma_k$ decreases monotonically to $0$ as $k$ increases, since the $f_k$ are decreasing in $k$. Additionally, a short calculation shows that $\gamma_1 = \pi^2/12,$ which is again compatible with the exponent of Hardy and Ramanujan's asymptotic formula for partitions into distinct parts. The corresponding enumeration function was denoted by $q(n)$ in \cite{HR18}, where they showed that
\begin{equation*}
q(n) \sim \frac{1}{4 \cdot 3^{\frac{1}{4}} n^{\frac{3}{4}}} e^{\pi \sqrt{\frac{n}{3}}}.
\end{equation*}

\end{remark}

\begin{remark}
In the case $k=2$ we find that $w_2 = \phi^{-1}$, where $\phi := (1 + \sqrt{5})/2$ is the golden ratio. Furthermore, the first and third special values on page 7 of \cite{Zag07} give the evaluation
\begin{align*}
g_2(v_2) & = \frac{1}{2} \left(\log{\phi}\right)^2 + \Li_2\left(\phi^{-1}\right) - \frac{1}{2} \Li_2\left(\phi^{-2}\right) \\
& = \frac{1}{2} \left(\log{\phi}\right)^2 + \frac{\pi^2}{10} - \left(\log{\phi}\right)^2
- \frac{\pi^2}{30} + \frac{1}{2}\left(\log{\phi}\right)^2 = \frac{\pi^2}{15}.
\end{align*}
Plugging in to the theorem statement, this gives
\begin{equation*}
c_2(n) \sim \frac{\sqrt{\phi}}{2 \cdot 3^{\frac{1}{4}} \sqrt{5} n^{\frac{3}{4}}} e^{2 \pi \sqrt{\frac{n}{15}}},
\end{equation*}
which was previously proven by Lehner in his study of the Rogers-Ramanujan products in \cite{Leh41}.
\end{remark}

The remainder of the paper is structured as follows. In Section \ref{S:notation} we give many basic identities for hypergeometric $q$-series and determine the critical points of the auxiliary functions $g_k$ and $h_k$. Section \ref{S:Double} contains analytic and combinatorial proofs of the double series representation from Theorem \ref{T:Ckdouble}, and also presents several combinatorial observations. We conclude with Section \ref{S:Asymp}, where we use the Constant Term Method and a Saddle Point analysis to prove the asymptotic formula from Theorem \ref{T:cknAsymp}.

\section{Hypergeometric series and auxiliary functions}
\label{S:notation}

In this section we recall several standard facts from the theory of hypergeometric $q$-series, including useful identities for special functions and modular transformations.

\subsection{Definitions and identities for $q$-series}
\label{S:notation:q}

The {\it dilogarithm} function \cite[p. 5]{Zag07} is defined for complex $|x| < 1$ by
\begin{equation*}
\Li_2(x) := \sum_{n \geq 0} \frac{x^n}{n^2}.
\end{equation*}
This function has a natural $q$-deformation that is known as the {\it quantum dilogarithm} \cite[p. 28]{Zag07}, which is given by ($|x|, |q| < 1$)
\begin{equation*}
\label{E:qLi2}
\Li_2(x; q) := -\log (x; q)_\infty = \sum_{n \ge 1} \frac{x^n}{n(1-q^n)}.
\end{equation*}
Moreover, an easy calculation shows that its Laurent expansion begins with the terms
\begin{align}
\label{E:TaylorLi}
\Li_2\left(x; e^{-\varepsilon}\right)
=\frac{1}{\varepsilon} \Li_2(x)-\tfrac12 \log(1-x)+O(\varepsilon),
\end{align}
where the series converges uniformly in $x$ as $\varepsilon \to 0^+$.

Next, we recall two identities due to Euler, which state that \cite[ equations (2.2.5) and (2.2.6)]{And98}
\begin{align}
\label{E:Euler1}
\frac{1}{(x;q)_\infty} & = \sum_{n \geq 0} \frac{x^n}{(q;q)_n}, \\
\label{E:Euler2}
(x;q)_\infty &= \sum_{n \geq 0} \frac{(-1)^n x^n q^{\frac{n(n-1)}{2}}}{(q;q)_n}.
\end{align}

Finally, Jacobi's {\it theta function} is defined by
\begin{equation}
\label{E:theta}
\theta(q; x) := \sum_{n \in \Z} q^{n^2} x^n.
\end{equation}
In order to determine the asymptotic behavior near $q = 1$, we use for $\varepsilon >0$ the modular inversion formula (cf. \cite[p. 290]{SS03}),
\begin{equation}
\label{E:thetainv}
\theta\left(e^{-\varepsilon}; e^{2 \pi i u}\right)
= \sqrt{\frac{\pi}{\varepsilon}} \sum_{n \in \Z} e^{-\frac{\pi^2 (n+u)^2}{\varepsilon}}.
\end{equation}

\subsection{Auxiliary functions}\label{S:notation:Aux}
We now prove several useful facts about the auxiliary functions  $h_k$ and $g_k$.

\begin{proposition}
\label{P:crit}
Adopt the above notation.
\begin{enumerate}
\item
\label{P:crit:h}
There is a unique root $w_k \in (0,1)$ of $h_k(x)$.
\item
\label{P:crit:g}
The unique critical point of $g_k$ on the positive real axis is given by $v_k$ such that $e^{2 \pi i v_k} = w_k$. Furthermore, $g''(v_k) < 0$.
\end{enumerate}
\end{proposition}

\begin{proof}
\ref{P:crit:h} Descartes' Rule of Signs implies that $h_k$ has either zero or two positive real roots. It is immediate to verify that $g_k(0) = 1, g_k(1) = 0$ and $g_k(3/4) < 0$ for $k \geq 2$, so the second root must lie in $(0,1)$ as claimed.

\ref{P:crit:g} Next, to identify the critical points of $g_k$, we calculate its derivative
\begin{equation*}
g_k'(u) = -4\pi^2 u - \log\left(1-e^{2 \pi i u}\right) 2\pi i +\log \left( 1-e^{2 \pi iku}\right) 2\pi i.
\end{equation*}
This vanishes precisely when
\begin{equation*}
2\pi i u + \log \left( \frac{1-e^{2 \pi i ku}}{1-e^{2 \pi iu}}\right) = 0.
\end{equation*}
Exponentiating and writing $x := e^{2 \pi i u}$ then shows that the critical points of $g_k(u)$ correspond to the roots of $h_k(x).$

Finally, we calculate the second derivative (again writing $x = e^{2 \pi i u}$) of $g_k$
\begin{align*}
g_k''(u) & = -4\pi^2 +\frac{(2\pi i)^2 x}{1-x} - \frac{(2\pi i)^2kx^k}{1-x^k}
= -4\pi^2 \left(\frac{1}{1-x} - \frac{kx^k}{1-x^k}\right).
\end{align*}
At the critical point this further simplifies, since $1- w_k^k = w_k^{-1} (1-w_k)$, which gives
\begin{align*}
g_k''(v_k) = \frac{-4\pi^2 \left( 1-kw_k^{k+1}\right)}{1-w_k}.
\end{align*}
We claim that at the critical point $1 - kw_k^{k+1} >0.$ Indeed, the derivative of $h_k$ is
\begin{equation}
\label{E:hk'}
h_k'(x) = (k+1)x^k -2,
\end{equation}
and at the root $w_k$, we have $h_k'(w_k)<0.$
Plugging in $w_k$ to  \eqref{E:hk'}, multiplying by $w_k$ and substituting $w_k^k = 2w_k - 1$ then implies that
\begin{equation*}
0 > (k+1)w_k^{k+1} - 2w_k = kw_k^{k+1} - 1.
\end{equation*}
This completes the proof of \ref{P:crit:g}.
\end{proof}

\section{Proof of Theorem \ref{T:Ckdouble}}
\label{S:Double}

\subsection{Analytic proof}
\label{S:Double:Analytic}

We follow Andrews' proof of Theorem 2 in \cite{And05}.
First we observe that $\calC_k$ satisfies the $q$-difference equation
\begin{equation}
\label{E:Ckqdiff}
\calC_k(z;q) = \sum_{j = 0}^{k-1} z^j q^{\frac{j(j+1)}{2}} \calC_k\left(zq^{j+1}; q\right).
\end{equation}
The terms on the right result from conditioning on the length of the sequence that begins with $1$. The $j=0$ term corresponds to the case where there is no $1$, and thus the smallest part is at least $2$; the other terms correspond to the case that there is a run $1, 2, \ldots, j$, and no $j+1$, so the next part is at least $j+2$. Applying \eqref{E:Ckqdiff} twice, we obtain the relation
\begin{align}
\calC_k(z;q) - zq \calC_k(zq;q)
= \calC_k(zq;q) & - z^k q^{\frac{k(k+1)}{2}} \calC_k\left(z q^{k+1};q\right).
\label{E:Ckqdiffshort}
\end{align}

Now consider the double series
\begin{equation*}
F_k (z;q) := \sum_{j,r \geq 0} \frac{(-1)^j z^{kj + r} q^{\frac{(r+kj)(r+kj+1)}2 +k \frac{j(j-1)}2}}{\left( q^k; q^k\right)_j (q;q)_{r}}.
\end{equation*}
Expanding this as a series in $z$, so that $F_k(z;q) =: \sum_{n \geq 0} \gamma_n(q) z^n$, we therefore have
\begin{equation*}
\gamma_n = \sum_{kj+r=n}
\frac{(-1)^j q^{\frac{r(r+1)}2 + krj+ \frac{k(k+1)j^2}2}}{\left( q^{k}; q^{k}\right)_j (q;q)_{r}}.
\end{equation*}
Now we calculate
\begin{align*}
\left( 1-  q^n \right) \gamma_n & =
\sum_{kj+r=n}
\frac{(-1)^j q^{\frac{r(r+1)}2 + krj+ \frac{k(k+1)}2 j^2}}{\left( q^{k}; q^{k}\right)_j (q;q)_{r}}
\Big( \left(1-q^r\right) + q^r \left( 1-q^{kj}\right)\Big) \\
& = q^n \gamma_{n-1} - q^{(k+1)(n-k) + \frac{k(k+1)}2} \gamma_{n-k},
\end{align*}
where the first term follows from the shift $r \mapsto r+1$, and the second term from $j \mapsto j+1$.
Multiplying by $z^n$ and summing over $n$ finally gives the $q$-difference equation
\begin{equation*}
F_k (z;q) = (1+zq) F_k (zq;q) - z ^k q^{\frac{k(k+1)}2} F_k \left( zq^{k+1};q\right).
\end{equation*}
As this is equivalent to \eqref{E:Ckqdiffshort}, we therefore conclude (cf. \cite{And75} and the uniqueness of solutions to $q$-difference equations) that $\calC_k = F_k$, completing the proof of Theorem \ref{T:Ckdouble}.

\subsection{Combinatorial proof}
\label{S:Double:Comb}

In this section we follow the approach from Section 3.2 of \cite{BLM13}, using a combinatorial decomposition of partitions into simple components that essentially split the double summation in Theorem \ref{T:Ckdouble}. Denote the size of a partition $\lambda$ by $|\lambda|$ and write $\ell(\lambda)$ for the number of parts, or {\it length}. Let $\calD_k$ be the set of partitions without $k$-sequences or repeated parts, and note that with this notation we have
\begin{equation*}
\calC_k(z;q) = \sum_{\lambda \in \calD_k} z^{\ell(\lambda)} q^{|\lambda|}.
\end{equation*}
If $\lambda \in \calD_k$ and $\ell(\lambda) = m$, so that $\lambda = \lambda_1 + \cdots + \lambda_m$ in nonincreasing order, then define $\lambda'$ by removing a triangular partition $(m-1) + (m-2) + \cdots + 1$, so that the new parts are
\begin{equation*}
\lambda'_j := \lambda_j - (m - j), \qquad 1 \leq j \leq m.
\end{equation*}
The definition of $\calD_k$ implies that $\lambda'$ is a partition in which each part occurs at most $k-1$ times, so
\begin{equation*}
\sum_{\lambda \in \calD_k} z^{\ell(\lambda')} q^{|\lambda'|}
= \prod_{n \geq 1} \left(1 + zq^n + z^2 q^{2n} + \cdots + z^{k-1} q^{n(k-1)}\right)
= \frac{\left(z^k q^k; q^k\right)_\infty}{(zq; q)_\infty}.
\end{equation*}
Euler's summation formulas (Corollary 2.2 in \cite{And98}) then imply the double series
\begin{equation}
\label{E:Dklambda'}
\sum_{\lambda \in \calD_k} z^{\ell(\lambda')} q^{|\lambda'|}
= \sum_{j, r \geq 0} \frac{(-1)^j z^{kj} q^{\frac{kj(j+1)}{2}}}{\left(q^k; q^k\right)_j} \frac{z^r q^r}{(q;q)_r}.
\end{equation}
To complete the proof, observe that
\begin{align*}
\ell(\lambda) & = \ell(\lambda'), \\
|\lambda| & = |\lambda'| + \frac{\ell(\lambda) (\ell(\lambda) + 1)}{2}.
\end{align*}
Plugging in to \eqref{E:Dklambda'}, we obtain
\begin{equation*}
\calC_k(z;q)
= \sum_{j, r \geq 0} \frac{(-1)^j z^{kj+r} q^{\frac{(kj+r)(kj+r-1)}{2} + \frac{kj(j+1)}{2} + r}}
{\left(q^k; q^k\right)_j (q;q)_r}.
\end{equation*}
Theorem \ref{T:Ckdouble} follows upon simplifying the exponent of $q$.

\begin{remark}
For example, if $k=3$ and $\lambda = 15 + 12 + 11 + 9 + 8 + 4 + 2 + 1$, then the associated $\lambda'$ is $8 + 6 + 6 + 5 + 5 + 2 + 1 + 1$, which consists of parts that are repeated at most twice.
\end{remark}

\subsection{Monotonicity}

\bigskip
We close with several additional combinatorial observations on the monotonicity of the enumeration functions. \begin{proposition}
\label{P:ineq}
For $m, n \geq 0$ and $k \geq 2$, we have
\begin{enumerate}
\item
\label{P:ineq:k}
$Q_k(m,n) \leq Q_{k+1}(m,n), $
\item
\label{P:ineq:n}
$Q_k(m,n) \leq Q_k(m, n+1).$
\end{enumerate}
\end{proposition}
\begin{proof}
As mentioned in the introduction, \ref{P:ineq:k} follows immediately from the definition. For \ref{P:ineq:n}, note that if $\lambda \in \calD_k$ is a partition of $n$ with $m$ parts, then
\begin{equation*}
(\lambda_1 + 1) +  \lambda_2 + \cdots + \lambda_m
\end{equation*}
is a partition of $n+1$ with $m$ parts. Furthermore, this new partition remains in $\calD_k$ since $\lambda_1 + 1 > \lambda_1 > \lambda_2 > \cdots > \lambda_m.$ 
\end{proof}

\begin{remark}
Part \ref{P:ineq:n} has the important consequence that
\begin{equation}
\label{E:Qknineq}
Q_k(n) \leq Q_k(n+1).
\end{equation}
A similar results hold for partitions without $k$-sequences: Lemma 10 in \cite{HLR04} states that $p_k(n) \leq p_k(n+1)$ (compare to \eqref{E:logpk}, noting that $\lambda_k$ are increasing in $k$). However, the analog to part \ref{P:ineq:n} is false in this case, as in general
\begin{equation*}
p_k(m,n) \not \leq p_k(m,n+1).
\end{equation*}
For example, the partitions without sequences of $2$ are $\{2, 1+1\}$, while the partitions of $3$ are $\{3, 1+1+1+\}$, so that $p_2(2,2) = 1$ and $p_2(2,3) = 0$.
\end{remark}

\section{Asymptotic Formulas}
\label{S:Asymp}

In this section, we study the asymptotic behavior of partitions without sequences or repeated parts, proving Theorem \ref{T:cknAsymp}. We first determine the asymptotic behavior of the generating function $\calC_k(q)$, and then deduce the asymptotic formula for its coefficients by applying Ingham's Tauberian Theorem.

\subsection{Constant Term Method and Saddle Point analysis}
\label{S:Asymp:Const}

We determine the asymptotic behavior of $\calC_k(q)$ near $q=1$ by using the Constant Term Method and a Saddle Point analysis. Throughout we restrict to real $q = e^{-\varepsilon}$ with $\varepsilon > 0$. The use of the Constant Term Method in the analytic study of $q$-series traces back to Meinardus \cite{Mein54}, and Nahm, Recknagel, and Terhoeven introduced the additional tools of asymptotic expansions and Saddle Point analysis \cite{NRT93}. In order to apply these techniques to double summation $q$-series, we follow the work of the first two authors in \cite{BHMV13}.

The main technical result of this analysis is an asymptotic formula for $\calC_k(q)$.
\begin{proposition}
\label{P:CkAsymp}
If $q = e^{-\varepsilon}$, then as $\varepsilon \to 0^+$ we have
\begin{equation*}
\calC_k (q) = \frac{2 \pi}{\sqrt{-g_k''(v_k)}} \left( 1+ O \left( \varepsilon^{\frac12}\right)\right) \exp \left( \frac{g_k(v_k)}{\varepsilon}\right).
\end{equation*}
\end{proposition}
\begin{proof}
We begin by rewriting the double series from Theorem \ref{T:Ckdouble} in the case $z=1$ as
\begin{align*}
\calC_k (q) & = q^{-\frac18} \sum_{j, r \geq 0} \frac{(-1)^j q^{\frac12 \left( r+kj+\frac12 \right)^2 + \frac{kj (j-1)}2}}{\left( q^k ; q^k \right)_j (q;q)_{r}} \\
& = \text{coeff}\left[x^0\right] \left( q^{-\frac18} \sum_{n\in\Z} x^{-n} q^{\frac12 \left( n+\frac12\right)^2}
\sum_{j\geq 0}\frac{(-1)^j x^{kj}q^{\frac{kj(j-1)}2}}{\left( q^k; q^k\right)_j}
\sum_{r\geq 0} \frac{x^r}{(q;q)_r}\right).
\end{align*}
The sums on $n, j,$ and $r$ can be expressed in terms of well-known functions using \eqref{E:theta}, \eqref{E:Euler2}, and \eqref{E:Euler1}, respectively. Plugging in the above definitions and applying Cauchy's Theorem, we obtain the integral representation
\begin{align*}
\calC_k (q)
& = \text{coeff}\left[x^0\right] \left( \theta\left( q^{\frac12}; x^{-1}q^{\frac12} \right) \exp \left( -\text{Li}_2 \left( x^k ; q^k \right) + \text{Li}_2 \left( x; q\right) \right) \right) \\
& = \int_{[0,1]+ic} \theta\left( q^{\frac12}; x^{-1}q^{\frac12} \right) \exp \left( -\text{Li}_2 \left( x^k ; q^k \right) + \text{Li}_2 \left( x; q\right) \right) du,
\end{align*}
where $c > 0$ is a constant that will be specified shortly.
Letting $\varepsilon \mapsto \varepsilon/2$ and $u \mapsto -u + i \varepsilon/(4 \pi)$ in \eqref{E:thetainv}, we obtain
\begin{align}
\label{E:CkthetaLi}
\calC_k (q) &= \sqrt{\frac{2\pi}{\varepsilon}} \sum_{n\in\Z} \int_{[0,1]+ic} \exp\left( -\frac{2\pi^2}{\varepsilon} \left( n-u +\frac{i\varepsilon}{4\pi}\right)^2 - \text{Li}_2 \left( x^k; q^k\right) + \text{Li}_2 \left( x; q\right)\right)du \notag \\
& = \sqrt{\frac{2\pi}{\varepsilon}} \int_{\R+ic} \exp \left( -\frac{2\pi^2}{\varepsilon} \left( u - \frac{i\varepsilon}{4\pi}\right)^2  - \text{Li}_2 \left( x^k ; q^k\right) + \text{Li}_2 \left( x; q \right)\right) du.
\end{align}
By \eqref{E:TaylorLi}, we have the asymptotic expansion
\begin{equation*}
\text{Li}_2 \left( x; q \right) - \text{Li}_2 \left( x^k; q^k \right) =
\frac{1}{\varepsilon}\left( \text{Li}_2 (x) - \frac{1}{k} \text{Li}_2 \left( x^k \right)\right)
 - \frac12 \log (1-x) + \frac12 \log \left( 1-x^k \right) + O(\varepsilon).
\end{equation*}

In order to perform a Saddle Point analysis, the leading $1/\varepsilon$ term from the exponent in \eqref{E:CkthetaLi} must be isolated, which gives the definition of the auxiliary function $g_k$. The overall exponent in the integrand can then be written as
\begin{equation}
\label{E:intexp}
\exp \left( \frac{g_k(u)}{\varepsilon} + \pi i u +\frac12 \log \left( 1-x^k \right) - \frac12 \log (1-x) + O(\varepsilon)\right).
\end{equation}
Proposition \ref{P:crit} implies that the asymptotic expansion of the integral is dominated by the critical point $v_k$, and the natural choice for the integration path is to set $c := \log (w_k^{-1} )/(2 \pi)$.

To conclude, we follow the standard argument by expanding the Taylor series around $v_k$ in \eqref{E:intexp}, using the change of variables $u = v_k + \sqrt{\varepsilon} z$. We thereby obtain
\begin{equation*}
\label{E:intAsymp}
\sqrt{\frac{1-w_k^k}{1-w_k}}\sqrt{w_k} \exp \left( \frac{g_k(v_k)}{\varepsilon} + \frac{g_k''(v_k)}2 z^2 + O \left( \varepsilon^{\frac12}\right) \right).
\end{equation*}
The terms outside of the exponential in the above expression simplify to $1$, so with the change of variables taken into account the integral becomes
\begin{align*}
\calC_k (q) & = \sqrt{\frac{2 \pi}{\varepsilon}} \left( 1+ O \left( \varepsilon^{\frac12}\right)\right) \exp \left( \frac{g_k(v_k)}{\varepsilon}\right) \sqrt{\varepsilon} \int_\R e^{\frac{g_k''(v_k)}2 z^2}dz \\
& = \frac{2 \pi}{\sqrt{-g_k''(v_k)}} \left( 1+ O \left( \varepsilon^{\frac12}\right)\right) \exp \left( \frac{g_k(v_k)}{\varepsilon}\right).
\end{align*}
The final equality follows from the Gaussian integral evaluation, which concludes the proof.
\end{proof}

\subsection{Ingham's Tauberian Theorem}
The asymptotic formula for partitions into distinct parts without sequences is now a consequence of the following Tauberian Theorem from \cite{Ing41}. This result describes the asymptotic behavior of the coefficients of a power series using its analytic behavior near the radius of convergence.

\begin{theorem}[Ingham]\label{T:Ingham}
Let $f(q)=\sum_{n\geq 0}a(n) q^n$ be a power series with weakly increasing nonnegative
coefficients and radius of convergence equal to $1$. If there are constants $A>0$ and
$\lambda, \alpha\in\R$ such that as $\varepsilon \to 0^+$ we have
\[
f \left(e^{-\varepsilon} \right)  \sim
\lambda   \varepsilon^\alpha
\exp\left(\frac{A}{\varepsilon}\right),
\]
then, as $n\to\infty$,
\[
a(n) \sim\frac{\lambda}{2\sqrt{\pi}}\,
\frac{A^{\frac{\alpha}{2}+\frac14}}{n^{\frac{\alpha}{2}+\frac34}}\,
\exp\left(2\sqrt{An}\right).
\]

\end{theorem}

\begin{proof}[Proof of Theorem \ref{T:cknAsymp}]

Proposition \ref{P:ineq} part \ref{P:ineq:n} implies that the coefficients are monotonically increasing (see \eqref{E:Qknineq}). We can therefore apply Theorem \ref{T:Ingham} to the asymptotic formula from Proposition \ref{P:CkAsymp} and directly obtain the stated asymptotic formula for $Q_k(n)$.
\end{proof}

\end{document}